\documentclass[12pt,reqno]{amsart}
\usepackage{graphicx}
\usepackage{amssymb,amsmath}
\usepackage{amsthm}
\usepackage{color,graphicx}
\usepackage{hyperref}
\usepackage{color}
\usepackage{epstopdf}
\usepackage{lpic}
\usepackage{booktabs}

\newcommand{\be}{\begin{equation}}

\newcommand{\ee}{\end{equation}}

\newcommand{\bigslant}[2]{{\raisebox{.1em}{$#1$}\left/\raisebox{-.1em}{$#2$}\right.}}

\numberwithin{equation}{section}
\numberwithin{figure}{section}

\newtheorem{theorem}{Theorem}[section]
\newtheorem{proposition}[theorem]{Proposition}
\newtheorem{remark}[theorem]{Remark}
\newtheorem{lemma}[theorem]{Lemma}
\newtheorem{corollary}[theorem]{Corollary}
\newtheorem{definition}[theorem]{Definition}

\begin{document}

\title[Regularized CH equation]{Periodic Waves for the Regularized Camassa-Holm Equation: Existence and Spectral Stability}

\author{F\'{a}bio Natali}
\address[F. Natali]{Departamento de Matem\'{a}tica - Universidade Estadual de Maring\'{a}, Avenida Colombo 5790, CEP 87020-900, Maring\'{a}, PR, Brazil}
\email{fmanatali@uem.br}



\begin{abstract}In this paper, we investigate the existence and spectral stability of periodic traveling wave solutions for the regularized Camassa-Holm equation. To establish the existence of periodic waves, we employ tools from bifurcation theory to construct solutions with the zero-mean property. We also prove that such waves may not exist for the well-known Camassa-Holm equation. Regarding spectral stability, we analyze the difference between the number of negative eigenvalues of the second variation of the Lyapunov functional at the wave, restricted to the space of zero-mean periodic functions, and the number of negative eigenvalues of the matrix formed from the tangent space associated with the low-order conserved quantities of the evolution model. Finally, we address the problem of orbital stability as a consequence of the spectral stability.
	
\end{abstract}

\date{\today}
\maketitle

\section{Introduction}

Consider the regularized Camassa-Holm (rCH) equation
\begin{equation}
\label{CH}
u_t+\omega u_x-u_{xxt}+3uu_x=2u_xu_{xx}+uu_{xxx},
\end{equation}
where $u:\mathbb{R}\times\mathbb{R}\rightarrow \mathbb{R}$ is a real-valued function and $\omega$ is a non-negative parameter and it is related to the critical shallow water wave speed. The model in \eqref{CH} can be seen as an abstract bi-Hamiltonian equation with infinitely many conservation laws (see \cite{CH} and \cite{fuch}) and can be viewed as a generalization of the well-known Camassa-Holm (CH) equation, in the sense that setting $\omega = 0$ in \eqref{CH}, we can recover the classical CH equation. In addition, the presence of the drift term $\omega u_x$ in equation \eqref{CH} introduces additional effects concerning the existence of smooth traveling wave solutions and it is also a parameter related to the critical shallow water speed (see \cite{Len3}). To explain better our purpose, let us construct some bridges between the classical CH and our model in $(\ref{CH})$. First, we need to set our problem:

In our paper, we consider equation (\ref{CH}) defined on the periodic domain $\mathbb{T} = [0,2\pi]$. In order to simplify the notation, we write 
$H^s_{\rm per}$ instead of $H^s_{\rm per}(\mathbb{T})$.
It is well known that the rCH equation (\ref{CH}) 
conserves formally the mass, momentum, and energy (see \cite{kalisch}) given by
\begin{equation}
\label{Mu}
M(u)=\int_0^{2\pi} u dx,
\end{equation}
\begin{equation}
\label{Eu}
E(u)=\frac{1}{2} \int_0^{2\pi} (u_x^2+u^2) dx,
\end{equation}
and
\begin{equation}
\label{Fu}
F(u)=\frac{1}{2} \int_0^{2\pi} (u^3+uu_x^2+\omega u^2) dx.
\end{equation}

\indent Some qualitative aspects have been established for the CH equation, that is, for $\omega = 0$ in $(\ref{CH})$ in the periodic context, and some of them can be easily adapted to show similar results for the rCH equation. Regarding the CH equation and its local well-posedness in periodic Sobolev spaces, the proofs are based on tools from semigroup theory and fixed-point arguments, and were established in \cite{CE-1998}, \cite{CE}, \cite{CE2000}, \cite{D}, \cite{hakka}, and \cite{HM2002}. In all these cases, it is possible to adapt the arguments to obtain similar results for the rCH equation. Sufficient conditions for the existence of smooth, peaked, and cusped periodic traveling waves associated with the full equation $(\ref{CH})$ were established in \cite{Len3}. Regarding the case $\omega = 0$ and the problem $(\ref{CH})$ posed in the whole real line, orbital stability results for peaked solitary waves in $H^1(\mathbb{R})$ were obtained in \cite{CM} and \cite{CS}. However, in the recent work \cite{NP}, we showed that perturbations to the peaked solitary waves actually grow in $W^{1,\infty}(\mathbb{R})$. Still in the case $\omega = 0$ in $(\ref{CH})$, but in the periodic setting, results on the orbital stability of peaked periodic waves in $H^1_{\rm per}$ were established in \cite{Len1} and \cite{Len2}. The orbital stability of the smooth periodic traveling waves in $H^1_{\rm per}$ was obtained in \cite{Len4} with the inverse scattering transform for initial data $u_0$ in $H^3_{\rm per}$ such that $m_0 = u_0 - u_0''$ is strictly positive. The spectral and orbital stability of smooth periodic waves for the classical CH equation were determined in \cite{GMNP}. To this end, we employed the analytic framework developed for the stability analysis of periodic waves in other nonlinear evolution equations of KdV type.

In \cite{kalisch} the author considered the model
 \begin{equation}
	\label{CHK}
	u_t+\omega u_x-u_{txx}+3uu_x=\gamma(2u_xu_{xx}+uu_{xxx}),
\end{equation}
posed over the unbounded domain $\mathbb{R}$ and proved the existence of solitary waves when $\omega\neq0$. In addition, if $\gamma<1$, the solitary wave is orbitally stable, and if $\gamma>1$, there exist both orbitally stable and unstable smooth solitary waves. To demonstrate this, the author employed the abstract approach in \cite{GrillakisShatahStraussI}. It is important to mention that the approach in \cite{GrillakisShatahStraussI} cannot be used to show orbital instability, only orbital stability, for the equation $(\ref{CH})$. The reason is that the Hamiltonian structure associated with equation $(\ref{CHK})$, given by $u_t = JG'(u)$, where $J = -(1 - \partial_x^2)^{-1} \partial_x$ and $G(u) = \frac{1}{2} \int_{\mathbb{R}} (u^3 + \gamma u u_x^2 + \omega u^2)dx$, is not suitable for applying the instability theorem in \cite{GrillakisShatahStraussI}, since $J$ is not an invertible operator with a bounded inverse in $H^1(\mathbb{R})$.

Let us describe our results. We seek for traveling waves of the form $u(x,t) = \phi(x-ct)$ with speed $c$ that
satisfy the third-order differential equation 
\begin{equation}
\label{third-order}
-c\phi'+c\phi'''+\omega\phi'+3\phi\phi'=2\phi'\phi''+\phi\phi'''
\end{equation}
After integration of (\ref{third-order}), we obtain the second-order differential equation
\begin{equation}\label{CHode}
-(c-\phi)\phi''+(c-\omega)\phi-\frac{3}{2}\phi^2+\frac{1}{2} \phi'^2 +A=0,
\end{equation}
where $A$ is the constant of integration. We consider smooth $2\pi$-periodic traveling wave solutions of \eqref{CH} with the zero-mean property 
which means that we are looking for solutions 
$\phi \in H^{\infty}_{\rm per}$ of the equation (\ref{CHode}) satisfying $\int_0^{2\pi}\phi dx=0$. Thus, the constant $A$ needs to satisfy
\begin{equation}\label{constA}
	A=\frac{1}{4\pi}\int_0^{2\pi}\phi'^2dx+\frac{3}{4\pi}\int_0^{2\pi}\phi^2dx.
	\end{equation}
	\indent To the best of our knowledge, since one of the physical motivations for equation $(\ref{CH})$ comes from shallow water wave theory, periodic waves with the zero-mean property may better describe water propagation than strictly positive waves. In addition, requiring that the average of $\phi$ is zero ensures that the total mass of water remains constant, that is, there is no net gain or loss of water as the wave travels at speed $c$ (see \cite[Section 4]{angulo}). From a mathematical standpoint, periodic waves with the zero-mean property do not have constant modes in their Fourier series expansions. This fact enables us to consider the traveling wave as a continuous curve of solutions depending only on the wave speed $c$, rather than a two-parameter continuous surface depending on both $c$ and the constant $A$ present in \eqref{CHode}. The existence of such a two-parameter surface has been reported in \cite{GMNP}, where the authors demonstrate the presence of fold points for the CH when the standard construction of periodic solutions connected to the first Hamiltonian structure is considered\footnote{There is another construction related to the second Hamiltonian structure, also reported in \cite{GMNP}. In this case, there are no fold points.}. To make clear for the readers; in our context, fold points are specific points $(c_0, A_0)$ in the parameter regime where solutions exist and such that the kernel of the second variation of the Lyapunov functional at the wave $\phi$ is two-dimensional. For the model in $(\ref{CH})$ with $\omega > 0$, fold points may occur. However, the simplicity of the kernel associated with the second variation of the Lyapunov functional at $\phi$ is not essential for our purposes in establishing the spectral and orbital instability of traveling waves.\\
\indent In order to prove the existence of periodic waves, we are going to use a different technique compared with some standard approaches in the current literature for the existence of periodic waves. First, we show the existence of small-amplitude periodic waves by the bifurcation theory established by Crandall-Rabinowitz theorem (see \cite[Chapter 8]{buffoni-toland}). The existence of small-amplitude periodic waves and a compactness argument enables us to extend the local solution to a global one in the sense that 
\begin{equation}\label{contcurve}c\in\left(\frac{\omega}{2},+\infty\right)\mapsto \phi\in H_{\rm per,m,e}^2\end{equation} exists and it is a continuous curve depending on $c$. Here, $H_{\rm per,m,e}^{s}$ denotes the Sobolev space constituted by periodic functions in $H_{\rm per}^s$ that are even and satisfy the zero-mean property. In addition, as we will see later on, the existence of a continuous curve is sufficient for our purposes. We do not need to assume any additional property regarding the smoothness of the mapping in $(\ref{contcurve})$ with respect to $c$ in order to obtain spectral and orbital stability, as was done, for instance, in \cite{GMNP}.

The Hamiltonian form for the rCH equation (\ref{CH}) is given by 
\begin{equation}
\label{sympl-1}
u_t = J F'(u), \quad 
J = -(1-\partial_x^2)^{-1} \partial_x, \quad 
F'(u) = \frac{3}{2} u^2 - u u_{xx} - \frac{1}{2} u_x^2+\omega u,
\end{equation}
where $J$ is a well-defined operator from $H^s_{\rm per}$ to $H^{s+1}_{\rm per}$ for every $s \in \mathbb{R}$ and $F'(u)$ is defined from $H^{s}_{\rm per}$ to $H^{s-2}_{\rm per}$ for $s > \frac{3}{2}$. The Cauchy problem associated with the problem (\ref{sympl-1}) is locally well-posed in the space $H_{\rm per}^s$ for $s>\frac{3}{2}$. This result is obtained using arguments similar to those established in the case of CH, as the proofs are based on semigroup theory and fixed point methods (see \cite{CE-1998}, \cite{D}, \cite{hakka}, and \cite{HM2002}).

The second order equation (\ref{CHode}) establishes that $\phi$ is a critical point of the Lyapunov functional given by 
\begin{equation}
\label{action-1}
\Gamma_{c}(u) = c E(u) - F(u) + A M(u).
\end{equation}
In addition, the second variation of the Lyapunov functional at the wave $\phi$, commonly called the linearized operator around the wave $\phi$, is then given by
\begin{equation}
\mathcal{L} = - \partial_x (c-\phi) \partial_x + (c-\omega - 3\phi + \phi''),
\label{hill} 
\end{equation} 
which is related to the action functional (\ref{action-1}) as $\mathcal{L}= \Gamma_{c}''(\phi)$. The linearized operator $\mathcal{L} : D(\mathcal{L})=H^2_{\rm per} \subset L^2_{\rm per} \mapsto L^2_{\rm per}$ is a self-adjoint, unbounded operator in $L^2_{\rm per}$ equipped with the standard inner product $\langle \cdot, \cdot \rangle$. 

\indent In our paper, to establish the spectral and orbital stability of the periodic wave $\phi$, we need to prove that the operator $\mathcal{L}$ in $(\ref{hill})$ has exactly two negative eigenvalues, both of which are simple. In addition, we also have to prove that $\ker(\mathcal{L})=[\phi']$. To do so, we employ some tools of the Floquet theory as in \cite{eastham}, \cite{Magnus}, and \cite{natali-neves}.

To start with spectral and orbital stability framework, let us consider the perturbation $v$ to the smooth traveling wave $\phi$ propagating with the same fixed speed $c$ given by
\begin{equation}
	\label{perturbation-v}
	u(x,t) = \phi(x-ct) + v(x-ct,t).
\end{equation}

Substituting the change of variables $(\ref{perturbation-v})$ into the equation $(\ref{CH})$, we obtain
\begin{equation}
	\label{CH-linearization}
	(1 - \partial_x^2) (v_t - cv_x) +\omega v_x +3 \partial_x (\phi v) + 2 v v_x = \partial_x (\phi v_{xx} + \phi' v_x + \phi'' v) + 2 v_x v_{xx} + v v_{xxx}.
\end{equation}
Neglecting the higher order terms in $v$, we obtain the linearized equation
\begin{equation}
	\label{CH-lin}
	v_t = - J \mathcal{L} v, 
\end{equation}
where $J$ is given by (\ref{sympl-1}) and $\mathcal{L}$ 
is the linearized operator defined in (\ref{hill}).

\begin{definition}\label{defstab-spectral}
	We say that the smooth periodic traveling wave $\phi \in H^{\infty}_{\rm per,m}$ is spectrally stable in the evolution problem (\ref{CH}) if the spectrum of $J \mathcal{L}$ in $L^2_{\rm per}$ 
	is located on the imaginary axis. 
\end{definition}

Another important mathematical reason for considering periodic solutions with the zero-mean property is as follows: since $J$ is not a one-to-one operator over $L_{\rm per}^2$, we need to consider it in a suitable subspace contained in $L_{\rm per}^2$. By restricting the spectral problem $J\mathcal{L}v=\lambda v$ in the space $L_{\rm per,m}^2$ constituted by periodic functions in $L_{\rm per}^2$ with the zero-mean property, we obtain a new spectral problem  
\begin{equation}\label{modspecp1}
J\mathcal{L}\big|_{L_{\rm per,m}^2}v=J\mathcal{L}_{\Pi}v=\lambda v, \end{equation}
where $J$ is one-to-one over $L_{\rm per,m}^2$ and $\mathcal{L}_{\Pi}$ is defined as 

\begin{equation}
	\mathcal{L}_{\Pi} = - \partial_x (c-\phi) \partial_x + (c-\omega - 3\phi + \phi'')+\frac{1}{2\pi}\langle\phi', \partial_x\cdot\rangle+\frac{3}{2\pi}\langle\phi, \cdot\rangle.
	\label{hillmeanzero} 
\end{equation} 

\indent Thus, the Definition $\ref{defstab-spectral}$ reads as follows in the new context.

\begin{definition}\label{defstab-spectralzm}
	We say that the smooth periodic traveling wave $\phi \in H^{\infty}_{\rm per,m}$ is spectrally stable in the evolution problem (\ref{CH}) if the spectrum of $J \mathcal{L}_{\Pi}$ in $L^2_{\rm per,m}$ 
	is located on the imaginary axis. 
\end{definition}
\begin{remark}
	Although Definition $\ref{defstab-spectralzm}$ can be used to general periodic solutions, not only to those with the zero-mean property, it provides a suitable connection between the wave and the functional space in which we are studying spectral stability: a periodic solution with the zero-mean property that is spectrally stable in $L_{\rm per,m}^2$.
\end{remark}
\indent As a consequence of the spectral stability, we have, in our case the orbital stability in the energy space $H_{\rm per}^1$ (respectively, $H_{\rm per,m}^1$).
\begin{definition}\label{defstab}
	We say that the periodic traveling wave $\phi \in H^{\infty}_{\rm per,m}$ is orbitally stable in the evolution problem (\ref{CH}) in $H_{\rm per}^1$ (respectively, $H_{\rm per,m}^1$), if for any $\varepsilon>0$ there exists $\delta>0$ such that for any $u_0 \in H^s_{\rm per}$ (respectively, $H_{\rm per,m}^s$) with $s > \frac{3}{2}$ satisfying 
	\begin{equation*}
		\|u_0-\phi\|_{H_{\rm per}^1}<\delta,
	\end{equation*}
the global solution $u \in C(\mathbb{R},H^s_{\rm per})$ (respectively, $C(\mathbb{R,}H_{\rm per,m}^s)$) with the initial data $u_0$ 
satisfies
	\begin{equation*}
	\inf_{l\in{\mathbb R}} \| u(t,\cdot) - \phi(\cdot+l) \|_{H_{\rm per}^1} <\varepsilon
	\end{equation*}
	for all $t\geq0$.
\end{definition}
\indent To prove that $\phi$ is spectrally stable in the sense of Definition \ref{defstab-spectralzm}, we need to study the behaviour of the non-positive spectrum associated with the linear operators $\mathcal{L}$ and $\mathcal{L}_{\Pi}$ in \eqref{hill} and $(\ref{hillmeanzero})$, respectively. Both information are crucial to use an index theorem contained Proposition 4.1 in \cite{Pel-book} in order to establish that $\mathcal{L}_{|_{\{1,\phi-\phi''\}^{\bot}}}\geq0$. This property ensures that the linear operator $\mathcal{L}_{|_{\{1,\phi-\phi''\}^{\bot}}}$ is non-negative, which implies the spectral stability since the Hamiltonian-Krein index is zero (see \cite[Theorem 5.2.11]{KP}). The orbital stability can be seen as an immediate consequence of spectral stability by using the recent approach in \cite[Sections 3 and 4]{ANP}.

The following theorem is the main result of this paper and also summarizes the objectives outlined in the previous paragraphs.

\begin{theorem}
	\label{theorem-stability}
	Let $c>\frac{\omega}{2}$ be fixed.\\
	(i) There exists 
	a continuous mapping $c\in \left(\frac{\omega}{2},+\infty\right) \mapsto \phi_c=\phi\in H_{\rm per,m}^{\infty}$ of $2\pi-$periodic functions that solves equation $(\ref{CHode})$ with constant $A$ given by $(\ref{constA})$.\\
	 (ii) The linear operator $\mathcal{L}_{\Pi}$ defined in $(\ref{hillmeanzero})$ admits one negative eigenvalue which is simple and a simple zero eigenvalue associated with the eigenfunction $\phi'$.\\
	(iii) If $\frac{d}{dc}E(\phi)>0$, the $2\pi$-periodic wave $\phi$ is spectrally and orbitally stable in $H_{\rm per,m}^1$ in the sense of Definitions $\ref{defstab-spectralzm}$ and $\ref{defstab}$, respectively.
\end{theorem}

\section{Existence of periodic traveling waves - Proof of Theorem $\ref{theorem-stability}-(i)$}

\label{sec-2}
In this section, we establish the existence of small-amplitude periodic waves associated with the equation \eqref{CHode}. After that, we show that the small-amplitude periodic waves can be extended to a global branch. In fact, we demonstrate that for all $c > \frac{\omega}{2}$, the local solutions can be extended to a continuous mapping $c \in \left( \frac{\omega}{2}, +\infty \right) \mapsto \phi \in H^2_{\rm per,m,e}$. This property is particularly important in our context, as we cannot guarantee, using the global bifurcation theory, the existence of periodic wave profiles $\phi$ that depend smoothly on the parameter $c > \frac{\omega}{2}$, as required by classical stability theories (see \cite{GrillakisShatahStraussI}). To do so, we rely on the local and global bifurcation theory developed in \cite[Chapters 8 and 9]{buffoni-toland}, respectively. As a first step, we present the following result, which corresponds to Theorem 4.1 in \cite{Pel-book} and will be useful for our purposes.

\begin{proposition}\cite{Pel-book}
	\label{prop-Pel} Let $L:D(L)\subset H\rightarrow H$ be a self-adjoint operator defined in a Hilbert space $H$ with the inner product $\langle \cdot, \cdot \rangle$ such that $L$ has $n(L)$ negative eigenvalues (counting their multiplicities) 
	and $z(L)$ the multiplicity of the zero eigenvalue bounded away from the positive spectrum of $L$. Let $\{ v_j \}_{j = 1}^N$ be a linearly independent set in $H$ and define
	$$
	H_0 = \{ f \in H;\ \{ \langle f, v_j \rangle = 0 \}_{j=1}^N\}.
	$$ 
	Let $\mathcal{A}(\lambda)$ be the matrix-valued function defined by its elements
	\begin{equation}\label{matrixA}
	A_{ij}(\lambda) = \langle (L - \lambda I)^{-1} v_i, v_j \rangle, \quad 1 \leq i,j \leq N, \quad \lambda \notin \sigma(L).
	\end{equation}
	Then, 
	\begin{equation}
		\label{count-neg}
		\left\{ \begin{array}{l}
			n(L \big|_{H_0}) = n(L) - n_0 - z_0, \\
			z(L \big|_{H_0})= z(L) + z_0 - z_{\infty},
		\end{array} \right.
	\end{equation}
	where $n_0$, $z_0$, and $p_0$ are the numbers of negative, zero, and positive eigenvalues of $\lim_{\lambda \uparrow 0} \mathcal{A}(\lambda)$ (counting their multiplicities) and $z_{\infty} = N - n_0 - z_0 - p_0$ is the number of eigenvalues of $A(\lambda)$ diverging in the limit $\lambda \uparrow 0$.
\end{proposition}
\begin{flushright}
	$\square$
\end{flushright}
\begin{remark}\label{pel} Some comments regarding Proposition $\ref{prop-Pel}$ in our context deserve to be highlighted. Let $\phi \in H_{\rm per,m,e}^{\infty}$ be a periodic traveling wave solution associated with the equation $(\ref{CHode})$. By using a standard planar analysis, we see that $\phi$ has only two zeroes over the interval $[0,2\pi)$, so the same behaviour occurs for $\phi'$. Consider $\mathcal{L}$ as the operator in $(\ref{hill})$ defined on $L_{\rm per}^2$, and suppose that $\phi < c$. The operator $\mathcal{L}$ can be rewritten as a Hill operator by applying the change of variables in $(\ref{changvar})$. The Floquet theory in \cite{eastham} and \cite{Magnus} can be used to conclude, from the fact that $\phi'$ has two zeroes over the interval $[0,2\pi)$ and $\mathcal{L} \phi' = 0$, that the eigenvalue $0$ is simple or double and corresponds to the second or third eigenvalue of $\mathcal{L}$. Consider $H_0$ in Proposition~$\ref{prop-Pel}$, defined as $H_0 = L_{\rm per,m}^2$. If $z(\mathcal{L}) \leq 1$, we immediately obtain $z_{\infty} = 0$, the matrix in $(\ref{matrixA})$ consists of just one entry given by $\langle\mathcal{L}^{-1}1,1\rangle$, and the values of $n_0$ and $z_0$ can be expressed, respectively, by
		\begin{equation}\label{z01}
		n_0=\left\{\begin{array}{lllll}
			1,\ \  {\rm if}\ \langle \mathcal{L}^{-1}1,1\rangle<0,\\
			0,\ \  {\rm if}\ \langle \mathcal{L}^{-1}1,1\rangle\geq0,
		\end{array}\right.\ \ \ {\rm and}\ \ \ z_0=\left\{\begin{array}{lllll}
		1,\ \  {\rm if}\ \langle \mathcal{L}^{-1}1,1\rangle=0,\\
		0,\ \  {\rm if}\ \langle \mathcal{L}^{-1}1,1\rangle\neq0.
		\end{array}\right.
			\end{equation}
Both values in $(\ref{z01})$ are essential to establish the existence of small-amplitude periodic waves in Proposition $\ref{propstokes1}$ and in the spectral stability analysis presented in Section 4.
\end{remark}

\indent To prove the existence of small-amplitude periodic waves, we first need some basic facts:
\begin{definition}\label{def1} (i) Let $H$ be a Hilbert space. An unbounded operator $T:D(T)\subset H\rightarrow H$ is a Fredholm operator if ${\rm{range}}(T)$ is closed and $z(T)$ and $c(T)$ are both finite. Here, $c(L)$ indicates the dimension of ${\rm coker}(T)$\footnote{Just to make clear for the readers: ${\rm coker}(T)$ denotes the quotient space given by ${\rm coker}(T)=\bigslant{H}{{\rm range}(T)}$.	}.\\
(ii) The index of an unbounded Fredholm operator $T:D(T)\subset H\rightarrow H$ is given by $ind(T)=z(T)-c(T)\in \mathbb{Z}$. A Fredholm operator is of index zero if $ind(T)=0$.	 
			\end{definition}
			
\begin{lemma}\label{lema3} Let $H$ be a real Hilbert space and $K\subset H$ a closed subspace. It follows that, $$\bigslant{H}{K}\cong K^{\bot},$$
			where the notation $A \cong B$ indicates that $A$ and $B$ are isomorphic. Therefore, if both $A$ and $B$ are finite dimensional, they have the same dimension.
			\end{lemma}	
			\begin{proof} Let us define $\Lambda:\bigslant{H}{K}\rightarrow K^{\bot}$ given by $\Lambda(\widetilde{u})=u-P_Ku$, where $P_K$ is the orthogonal projection from $H$ onto the closed subspace $K$, and $\bigslant{H}{K}$ are formed by cosets $\widetilde{u}=u+K$. It is well known that for any $u\in H$, we obtain $P_Ku\in K$ and $u-P_Ku\in K^{\bot}$, that is, $\Lambda$ is well-defined. In addition, since $||\Lambda(\widetilde{u})||_H=||u-P_Ku||$, we obtain by Pythagorean theorem  $||u||_{H}^2=||P_Ku||_H^2+||u-P_Ku||_H^2=||P_Ku||_H^2+||\Lambda(\widetilde{u})||_H^2.$
				The equality implies $||\Lambda(\widetilde{u})||_H^2=||u||_{H}^2-||P_Ku||_H^2\leq ||u||_{H}^2$, and thus, $\Lambda$ is a bounded operator. $\Lambda$ is an one-to-one operator since for $\Lambda(\widetilde{u})=0$, we have $u=P_Ku$ and this fact automatically implies $u\in K$, that is, $\widetilde{u}=0$. To see that $\Lambda$ is onto, we consider $v\in K^{\bot}$. By the definition of orthogonal projection from $H$ onto the closed subspace $K$, there exists $u\in H$ such that $v=u-P_Ku$, and $\Lambda$ is onto as desired.
				
			\end{proof}
			
			\begin{remark}\label{obs2} We can offer a new perspective on Definition $\ref{def1}$ for a Hilbert space $H$ and an unbounded self-adjoint linear operator $L: D(L) \subset H \rightarrow H$ with closed range. In fact, since $L$ is self-adjoint with closed range, Lemma $\ref{lema3}$ implies that $\bigslant{H}{{\rm range}(L)}= \bigslant{H}{\ker(L)^{\bot}}\cong\ker(L)^{{\bot}{\bot}}=\ker(L)$. Therefore, if $z(L)$ is finite, we can conclude that $L$ is always a Fredholm operator of index zero.				
			\end{remark}
\indent We prove the existence of small-amplitude periodic waves in the next result. 
\begin{proposition}\label{propstokes1}
		There exists $a_0 > 0$ such that for all $a \in (0,a_0)$ there is an even local periodic solution $\phi$ for the problem \eqref{CHode}. The small-amplitude periodic waves are given by the following expansion:
	\begin{equation}\label{varphi-stokes31}
		\phi(x) = a \cos(x) + \frac{a^2}{\omega}\cos(2x) +\mathcal{O}(a^3).
	\end{equation}
	The wave speed $c$ and the constant of integration $A$ in $(\ref{constA})$ in this case are expressed as	
	
	\begin{equation}\label{varphi-stokes3}
		c = \frac{\omega}{2} +\frac{6a^2}{\omega}  + \mathcal{O}(a^4)\ \ \ \ \ \mbox{and}\ \ \ \ \  A=a^2+\mathcal{O}(a^4).
	\end{equation}
	\end{proposition}
	
\begin{proof} We shall give the steps how to prove the existence of small-amplitude periodic waves using \cite[Chapter 8]{buffoni-toland}. In fact, let $\mathsf{F}: H_{\rm per,m,e}^{2} \times (\tfrac{\omega}{2},+\infty) \rightarrow L_{\rm per,m,e}^2$ be the smooth map defined by
\begin{equation}\label{F-lyapunov}
	\mathsf{F}(g,r) = -(r-g)g''+(r-\omega)g-\frac{3}{2}g^2+\frac{1}{2}g'^2+\frac{1}{4\pi}\int_0^{2\pi}g'^2dx+\frac{3}{4\pi}\int_0^{2\pi}g^2dx,
\end{equation}
where we recall that  $H_{\rm per,m,e}^{s}$ indicates the Sobolev space constituted by periodic functions in $H_{\rm per}^s$ that are even and satisfy the zero-mean property. We see that $\mathsf{F}(g,r) = 0$ if, and only if, $g \in H_{per,m,e}^{2}$ satisfies \eqref{CHode} with corresponding wave speed $r \in (\tfrac{\omega}{2}, +\infty)$. The Fr\'echet derivative of the function $\mathsf{F}$ with respect to the first variable at the fixed point $(0,r_0)$ is then given by
\begin{equation}\label{Dg}
	D_g \mathsf{F}(0, r_0) f = (-r_0\partial_x^2+(r_0-\omega)) f.
\end{equation}

The nontrivial kernel of $D_g\mathsf{F}(0, r_0)$ is determined by functions $h \in H_{\rm per,m,e}^{2}$ such that
\begin{equation}
	\widehat{h}(k) (r_0-\omega+r_0k^2) = 0,
\end{equation}
where $\widehat{h}$ indicates the Fourier transform of $h$ with frequency $k$ in the periodic setting. We see that $D_g \mathsf{F}(0, r_0)$ has the one-dimensional kernel if, and only if, $r_0 = \frac{\omega}{1+k^2}$ for some $k \in \mathbb{Z}$. In this case, we have
$
	\ker(D_g \mathsf{F}(0, r_0)) = [\tilde{\varphi}_k],
$
where $\tilde{\varphi}_k(x) = \cos(kx)$. In addition, since $D_g \mathsf{F}(0, r_0)$ is a self-adjoint operator on $L^2_{\rm per, m, e}$ with domain in $H^2_{\rm per, m, e}$, the transversality condition $$(-\partial_x^2+1)(\cos(kx))\notin  \ker(D_g \mathsf{F}(0, r_0))^{\bot}={\rm range}D_g \mathsf{F}(0, r_0),$$ is also satisfied. 

 Next, we define the set $\mathcal{S} = \{(g, r) \in U;\ \mathsf{F}(g, r) = 0\},$ 
where $$U=\left\{(g,r)\in H_{\rm per,m,e}^2\times \left(\frac{\omega}{2},+\infty\right);\ g<r\right\}.$$
\indent  Let $(g, r) \in U$ be a real solution of $\mathsf{F}(g,r) = 0$. We want to show that the linear operator $\mathcal{P}_{\Pi}:L_{\rm per,m,e}^2\rightarrow L_{\rm per,m,e}^2$ with domain $D(\mathcal{P}_{\Pi})=H_{\rm per,m,e}^2$ and defined by
\begin{equation}\label{fredalt}
	\mathcal{P}_{\rm \Pi} h = D_g \mathsf{F}(g, r) h = - \partial_x (r-g) \partial_xh + (r-\omega - 3g + g'')h+\frac{1}{2\pi}\langle g',\partial_x h\rangle+\frac{3}{2\pi}\langle g, h\rangle,\end{equation}
is a Fredholm operator of index zero. Indeed, we first observe that the operator
$\mathcal{P} = - \partial_x (r - g) \partial_x + (r - \omega - 3g + g'')$
defined in $L_{\rm per,e}^2$, with domain $H_{\rm per,e}^2$, can be rewritten as a Hill operator by applying a change of variables similar to that in $(\ref{changvar})$, with $g$ in place of $\phi$. The Floquet theory presented in \cite{eastham} and \cite{Magnus} guarantees that the two possible periodic solutions of the equation $\mathcal{P} f = 0$, when $\mathcal{P}$ is defined on the entire space $L_{\rm per}^2$, are $g'$ (odd) and $\mathsf{y}$ (even). Therefore, when $\mathcal{P}$ is restricted to the space $L_{\rm per,e}^2$, we have $z(\mathcal{P}) \leq 1$. In addition, the function $\mathsf{y}$ which is even, may not be periodic and $\{\mathsf{y},g'\}$ is a fundamental set of solutions for the formal equation $\mathcal{P}f=0$. Since $(g, c) \in U$ is a solution of the equation $\mathsf{F}(g, c) = 0$, we immediately see that $g$ is even, and hence $g'$ cannot be considered an element of $\ker(\mathcal{P}_\Pi)$. By Proposition $\ref{prop-Pel}$, we have the relation $z(\mathcal{P}_\Pi) = z(\mathcal{P}) + z_0 - z_\infty$. Since $z(\mathcal{P}) \leq 1$, it follows that $z_\infty = 0$, and by Remark $\ref{pel}$, it follows that $z(\mathcal{P}_\Pi) = z(\mathcal{P}) + z_0 \leq 2$. Hence, the dimension of the kernel of $\mathcal{P}_{\Pi}$ is finite. Since $\mathcal{P}_{\Pi}$, defined in $L_{\rm per,m,e}^2$, is a self-adjoint operator with closed range, it follows by Remark $\ref{obs2}$ that $\mathcal{P}_{\Pi}$ is a Fredholm operator of index zero.\\
\indent The local bifurcation established by Crandall-Rabinowitz theorem (see \cite[Chapter 8]{buffoni-toland} and the beginning of Chapter 9 in \cite{buffoni-toland} for a more suitable explanation) guarantees the existence of an open interval $I$ containing $r_0 > \tfrac{\omega}{2}$, an open ball $B(0,\alpha) \subset H_{\rm per,m,e}^{2}$ for some $\alpha>0$ and a  smooth mapping
$r \in I \mapsto \varphi= \varphi_r \in B(0,\alpha) \subset H_{\rm per,m,e}^{2}$
such that $\mathsf{F}(\varphi,r) = 0$ for all $\omega \in I$ and $\varphi\in B(0,\alpha)$.

For each $k \in \mathbb{N}$, the point $(0, \tilde{r}_k)$ where $\tilde{r}_k = \frac{\omega}{1+k^2}$ is a bifurcation point. Moreover, there exists $a_0 > 0$ and a local bifurcation curve
\begin{equation}\label{localcurve}
	a \in (0,a_0) \mapsto (\varphi_{k,a}, r_{k,a}) \in H_{\rm per,m,e}^{2} \times (0,+\infty)
\end{equation}
which emanates from the point $(0, \tilde{r}_k)$ to obtain small-amplitude even $\frac{2\pi}{k}$-periodic solutions with the zero-mean property for the equation \eqref{CHode}. In addition, we have $r_{k,0} = \tilde{r}_k$, $D_a \varphi_{k,0} = \tilde{\varphi}_k$ and all solutions of $\mathsf{F}(g, r) = 0$ in a neighborhood of $(0, \tilde{r}_k)$ belongs to the curve in $(\ref{localcurve})$ depending on $a \in (0,a_0)$.\\
\indent Finally, let us consider the case $k = 1$, since we are interested in $2\pi$-periodic solutions. Define in $(\ref{localcurve})$ the functions $\phi = \varphi_{1,a}$ and $c = r_{1,a}$. To obtain the expression in $\eqref{varphi-stokes3}$, we can use the Stokes expansions:
\begin{equation}\label{stokes}
\phi(x)=\sum_{n=1}^{+\infty}\phi_n(x)a^n\ \ \ \ \ {\rm and}\ \ \ \ \ \ c=\frac{\omega}{2}+\sum_{n=1}^{+\infty}c_{2n}a^{2n}.
	\end{equation}
where $\phi_1(x) = \cos(x)$ is commonly referred to as the generator of the small-amplitude periodic wave.  Substituting the ansatz in $(\ref{stokes})$ into equation $(\ref{CHode})$, and using the balance of coefficients corresponding to the powers of $a^n$, we obtain that $\phi_2(x) = \frac{1}{\omega} \cos(2x)$ and $c_2 = \frac{6}{\omega}$. In addition, by substituting the expressions for $\phi$ and $c$ from $(\ref{stokes})$ into the constant $A$ given by $(\ref{constA})$, we obtain $A = a^2 + \mathcal{O}(a^4)$.

\end{proof}


	

\indent Next, we extend the local solutions obtained in Proposition $\ref{propstokes1}$ to determine global solutions $\phi$ of equation $(\ref{CHode})$, in terms of the parameter $c$, for all $c > \frac{\omega}{2}$.
\begin{proposition}\label{propstokes2}

	The local solution obtained in Proposition $\ref{propstokes1}$ is global, that is, $\phi$ exists for all $c>\frac{\omega}{2}$. In addition, the pair $(\phi, c) \in U $ is continuous in terms of the parameter $c > \frac{\omega}{2}$ and it satisfies \eqref{CHode}.
\end{proposition}

\begin{proof}
	 To obtain that the  local curve \eqref{localcurve} extends to a global one, we need to prove  that every bounded and closed subset $\mathcal{R}\subset\mathcal{S}$ is a compact set on $H_{\rm per,m,e}^{2}\times (\tfrac{\omega}{2},+\infty)$. To this end, we want to prove that $\mathcal{R}$ is sequentially compact, that is, if $\{(g_n,c_n)\}_{n\in\mathbb{N}}$ is sequence in $\mathcal{R}$, then there exists a subsequence of $\{(g_n,c_n)\}_{n\in\mathbb{N}}$ that converges to a point in $\mathcal{R}$. Let $\{(g_n,c_n)\}_{n\in\mathbb{N}}$ be a sequence in $\mathcal{R}$. We obtain a subsequence with the same notation such that $c_n\rightarrow c$ in $\left[\frac{\omega}{2},+\infty\right),$
	and
	$g_n\rightharpoonup g$ in $H_{\rm per,m,e}^2,$ as $n\rightarrow +\infty$. If $c=\frac{\omega}{2}$, we obtain from the expression for $c_n=\frac{\omega}{2}+\frac{6a_n^2}{\omega}+\mathcal{O}(a_n^4)$ in a neighbourhood of $\frac{\omega}{2}$ to the right that $a_n\rightarrow 0$ as $n\rightarrow +\infty$. Therefore, the solution $g_n$ has the form in $(\ref{varphi-stokes31})$ for each $n\in\mathbb{N}$, and it satisfies $g_n\rightarrow 0$ in $H_{\rm per,m,e}^2$. Hence, the result is proved, but the zero solution is not interesting for our purposes. Now, if $c>\frac{\omega}{2}$, we automatically obtain that $g\neq c$ since $g$ has satisfies the zero-mean property. Thus, $(g_n,c_n)\in \mathcal{S}$ implies that $g_n<c_n$ and 
	\begin{equation}\label{estI}
	g_n''=\frac{c_n-\omega}{c_n-g_n}g_n-\frac{3}{2(c_n-g_n)}g_n^2+\frac{1}{2(c_n-g_n)}g_n'^2+\frac{A_n}{c_n-g_n},
	\end{equation}
	where $A_n=\frac{1}{4\pi}\int_0^{2\pi}g_{n}'^2dx+\frac{3}{4\pi}\int_0^{2\pi}g_{n}^2dx.$ The right-hand side of $(\ref{estI})$ is a bounded element in $H_{\rm per,e}^1$, that is, $g_n''$ is bounded in $H_{\rm per,e}^1$. Since $g_n\in H_{\rm per,m,e}^2$ for all $n\in\mathbb{N}$, we deduce that $(g_n)_{n\in\mathbb{N}}$ is a bounded sequence in $H_{\rm per,m,e}^3$, and by the compact embedding $H_{\rm per,m,e}^3\hookrightarrow H_{\rm per,m,e}^2$ we obtain, modulus a subsequence, that $g_n\rightarrow g$ in $H_{\rm per,m,e}^2.$
	In other words, $\mathcal{R}$ is compact in $H_{\rm per,m,e}^2$ as requested.
	 
\indent Since the wave speed $c$ of the wave given by $(\ref{varphi-stokes3})$ is not constant, we can apply \cite[Theorem 9.1.1]{buffoni-toland} to extend globally the local bifurcation curve given in \eqref{localcurve}. More precisely, there is a continuous mapping
\begin{equation}\label{globalcurve}
	c \in \left( \tfrac{\omega}{2}, +\infty \right) \mapsto \phi \in H_{\rm per,m,e}^{2}
\end{equation}
where $\phi$ solves the equation $(\ref{CHode})$.
\end{proof}

\begin{remark}
According to Propositions $\ref{propstokes1}$ and $\ref{propstokes2}$, we can deduce that in the case of the classical CH equation, that is, $\omega=0$ in equation $(\ref{CHode})$, we do not have a continuous mapping $c \in \left( \tfrac{\omega}{2}, +\infty \right) \mapsto \phi \in H_{\rm per,m,e}^{2}$ of periodic waves for the CH equation with a fixed period. Using the arguments in \cite[Section 2]{GMNP}, we can deduce that, for fixed values of $c_0$ and $A_0$ in $(\ref{CHode})$, there is a single solution with the zero-mean property for the case $\omega=0$, not a continuous curve that emanates from the equilibrium solution as determined above.
\end{remark}


\section{The linearized operator - Proof of Theorem $\ref{theorem-stability}-(ii)$}
\label{sec-3}

\subsection{The non-positive spectrum of $\mathcal{L}_{\Pi}$.}

\indent To start with the spectral analysis of the linear operator $\mathcal{L}_{\Pi}$, we first need a basic lemma.

\begin{lemma}
	\label{lemma-L-K}
	The spectrum of $\mathcal{L}$ defined in $L^2_{\rm per}$ with domain $H_{\rm per}^2$ is purely discrete. 
\end{lemma}

\begin{proof}
	Since $c - \phi > 0$ and $\phi \in H^{\infty}_{\rm per,m, e}$, 
	the linearized operator $\mathcal{L}$ with 
	the dense domain $H_{\rm per}^2 \subset L^2_{\rm per}$ 
	is a self-adjoint, unbounded operator in $L_{\rm per}^2$. 
	Consequently, $\sigma(\mathcal{L}) \subset \mathbb{R}$ is purely discrete 
	in $L^2_{\rm per}$ due to the compact embedding of $H^2_{\rm per}$ into $L^2_{\rm per}$.
	
	\end{proof}

The next three results describe the non-positive part 
of the spectrum of $\mathcal{L}$ in $L^2_{\rm per}$. 
The proofs rely on Theorem 3.1 in \cite{neves} 
(see also the classical Floquet theory in \cite{eastham} and \cite{Magnus}), on Sylvester's inertial law theorem in \cite[Theorem 2.2]{lopes} and on Theorem 3.1 in \cite{natali-neves}.

\begin{proposition}\cite{neves}
	\label{teo12}
	Let $\mathcal{M}_\tau = -\partial_x^2+Q(\tau,x)$
	be the Schr\"{o}dinger operator 
	with the even, $2\pi-$periodic, smooth potential $Q=Q(\tau,x)$, where $\tau=(\tau_1,\tau_2)$ is a pair  defined in an open subset $\mathcal{V}\subset \mathbb{R}^2$. Assume that 
	$\mathcal{M}_\tau w = 0$ is satisfied by a linear combination 
	of two solutions $\varphi_1$ and $\varphi_2$ satisfying
	$
	\varphi_1(x+2\pi) = \varphi_1(x) + \theta \varphi_2(x),
	$
	and 
	$
	\varphi_2(x+2\pi) = \varphi_2(x),
	$
	with some $\theta \in \mathbb{R}$. Assume that $\varphi_2$ has two zeros on the period of $Q$. The zero eigenvalue of $\mathcal{M}_\tau$ in $L^2_{\rm per}$ is simple if $\theta \neq 0$ and double if $\theta = 0$. It is the second eigenvalue of $\mathcal{M}_\tau$ if $\theta \geq 0$ and the third eigenvalue of $\mathcal{M}_\tau$ if $\theta < 0$. 
\end{proposition}
\begin{flushright}
	$\square$
\end{flushright}

\begin{proposition}\cite{lopes}
	\label{prop-Lopes} Let $L$ be a self-adjoint operator in a Hilbert space $H$ and $S$ be a bounded invertible operator in $H$. Then, $S L S^*$ and $L$ have the same inertia, that is, the dimensions of the negative, null, and positive subspaces of $H$ are the same.
\end{proposition}
\begin{flushright}
	$\square$
\end{flushright}

We now characterize the non-positive spectrum of $\mathcal{L}$ by using the following results below.

\begin{proposition}
	\label{theolinop1}
	Let $c>\frac{\omega}{2}$ be fixed. Consider the linearized operator $\mathcal{L} : D(\mathcal{L})=H^2_{\rm per} \subset L^2_{\rm per} \to L^2_{\rm per}$ as in $(\ref{hill})$. The spectral problem $\mathcal{L}v=\lambda v$ can be written as the weighted spectral problem $\mathcal{M}_{\tau}w=\lambda(c-\phi)^{-1} w$, where $\tau=(c,\omega)$ and $\mathcal{M}_{\tau}=-\partial_x^2+Q(\tau,x)$ is a Hill operator with a $\text{smooth, even, and $2\pi-$periodic}$ potential $Q(\tau,x)$. If the set $\{ y_1,y_2\}$ is a fundamental set for the equation $\mathcal{L} v=0$, we obtain that 
	\begin{equation}
		\label{rel-varphi22}
		\{ \varphi_1, \varphi_2 \} = \left\{\left(\frac{c-\phi(0)}{c-\phi}\right)^{-1/2} y_1,\left(\frac{c-\phi(0)}{c-\phi}\right)^{-1/2} y_2\right\}
	\end{equation}
	is the fundamental set of solutions associated with the equation $\mathcal{M}_\tau w=0$. In addition, the spectral problem $\mathcal{M}_{\tau}w=\lambda(c-\phi)^{-1} w$ has the same inertia as the linear operator $S\mathcal{M}_{\tau}S$, where $S=(c-\phi)^{1/2}$.

\end{proposition}

\begin{proof}
Our approach in this proposition is based on \cite[Theorem 4]{GMNP}. In order to transform the spectral problem $\mathcal{L}v = \lambda v$ into a convenient spectral problem involving  the Schr\"odinger operator $\mathcal{M}_{\tau}$, as stated in Proposition \ref{teo12}, we set $\tau =  (c, \omega)$. We write $\mathcal{L} v = \lambda v$ as the second-order differential equation
	\begin{equation}
		\label{perspectprob}
		p_1(x)v''+p_2(x)v'+(p_3(x)+\lambda)v=0, 
	\end{equation}
	with $p_1(x) = c-\phi(x)$, $p_2(x) = -\phi'(x)$, and $p_3(x) = -\phi''(x) + 3 \phi(x) - c+\omega$. The Liouville transformation
	\begin{equation}\label{liouv}
		D(x)=-\int_0^x\frac{\phi'(s)}{c-\phi(s)}ds=\ln\left(\frac{c-\phi(x)}{c-\phi(0)}\right)
	\end{equation}
	is nonsingular since $c - \phi > 0$.
	This last fact enables us to use the following change of variables 
	\begin{equation}\label{changvar}
		v(x) = w(x) e^{-\frac{1}{2} D(x)} = w(x)\sqrt{\frac{c-\phi(0)}{c-\phi(x)}}.
	\end{equation}
	into the second-order equation (\ref{perspectprob}) to obtain weighted spectral problem
	\begin{equation}
		\label{hilleq}
		-w''(x)+Q(\tau,x)w(x)=\lambda (c - \phi(x))^{-1} w(x),
	\end{equation}
	where 
	\begin{equation}\label{perpot}
		Q(\tau,x) = \frac{c-\omega-3\phi(x)}{c - \phi(x)} + \frac{\phi''(x)}{2(c - \phi(x))} +\frac{1}{4}\left(\frac{\phi'(x)}{c-\phi(x)}\right)^2. 
	\end{equation}
	
	The operator $\mathcal{M}_\tau$ satisfies the condition of Proposition \ref{teo12} 
	since $Q$ defined in $(\ref{perpot})$ is even, $2\pi$-periodic, and smooth. Therefore, if the set $\{ y_1,y_2\}$ is a fundamental set for the equation $\mathcal{L} v=0$, we obtain that 
	\begin{equation}
		\label{rel-varphi}
		\{ \varphi_1, \varphi_2 \} = \left\{\left(\frac{c-\phi(0)}{c-\phi}\right)^{-1/2} y_1,\left(\frac{c-\phi(0)}{c-\phi}\right)^{-1/2} y_2\right\},
	\end{equation}
	is the fundamental set of solutions associated with the equation $\mathcal{M}_\tau w=0$. In addition, $\varphi_1$ and $\varphi_2$ are related to each other through the equality
	\begin{equation}\label{relpq1}
		\varphi_1(x+2\pi) = \varphi_1(x) + \theta \varphi_2(x).
	\end{equation} 
	\indent Finally, let $S = (c-\phi)^{1/2}$ be a bounded and invertible multiplication operator defined on $L_{\rm per}^2$. With the transformation $w = (c-\phi)^{1/2}\widetilde{w}$, the spectral problem $(\ref{hilleq})$ can be rewritten as the spectral problem for the operator $S\mathcal{M}_{\tau}S$, which, by Proposition $\ref{prop-Lopes}$, has the same inertia as the Schr\"odinger operator $\mathcal{M}_{\tau} = -\partial_x^2 + Q(\tau,x)$, as required.

\end{proof}

\begin{proposition}\label{propR0}
	Let $c>\frac{\omega}{2}$ be fixed. Consider the linearized operator $\mathcal{L}_{\Pi} : D(\mathcal{L}_{\Pi})=H^2_{\rm per,m} \subset L^2_{\rm per,m} \to L^2_{\rm per,m}$ as in $(\ref{hillmeanzero})$. Then,  $\ker(\mathcal{L}_{\Pi})=[\phi']$ and $n(\mathcal{L}_{\Pi})=1$.
\end{proposition}
\begin{proof}
First, we see by $(\ref{CHode})$ and $(\ref{constA})$ that $\mathcal{L}_{\Pi}\phi=c(\phi''-\phi)+\omega\phi.$
Thus,
\begin{equation}\label{quadraPI}
	\langle\mathcal{L}_{\Pi}\phi,\phi\rangle=-c\int_0^{2\pi}\phi'^2dx-c\int_0^{2\pi}\phi^2dx+\omega\int_0^{2\pi}\phi^2dx.
\end{equation}
By applying the Poincar\'e–Wirtinger inequality to the first term on the right-hand side of equality 
$(\ref{quadraPI})$, we deduce that
\begin{equation}\label{quadraPI1}
	\langle\mathcal{L}_{\Pi}\phi,\phi\rangle\leq(-2c+\omega)\int_0^{2\pi}\phi^2dx.
\end{equation}
From inequality $(\ref{quadraPI1})$, together with the condition $c > \frac{\omega}{2}$ and the standard min–max theorem, it follows that $\mathcal{L}_{\Pi}$ has at least one negative eigenvalue, that is, we have $n(\mathcal{L}_{\Pi})\geq1$.\\
 \indent Next, since $\mathcal{M}_{\tau}$ is a Hill operator with even smooth potential $Q(\tau,x)$ and $\phi'$ has two zeroes over the interval $[0,2\pi)$, we obtain by Floquet theory (see for instance \cite{eastham} and \cite{Magnus}) that $\varphi_2$ has the same property, and therefore the zero eigenvalue of $\mathcal{M}_{\tau}$ is the second or the third eigenvalue. This implies, by Proposition $\ref{theolinop1}$, that $1 \leq n(\mathcal{L}) \leq 2$. Assume now that $n(\mathcal{L}) = 1$. By Proposition $\ref{prop-Pel}$, it follows, since $n(\mathcal{L}_{\Pi})\geq1$, that $n(\mathcal{L}_{\Pi})=n(\mathcal{L})-n_0-z_0=1-0-0=1$. In particular, $n_0 = z_0 = 0$. Even if $z_{\infty} = 1$, which implies $z(\mathcal{L}) = 2$, we note that, since $\phi'$ has the zero-mean property and $\mathcal{L}_{\Pi}\phi'=\mathcal{L}\phi' = 0$, it follows that $[\phi'] \subset \ker(\mathcal{L}_{\Pi})$. Thus, we automatically obtain $z(\mathcal{L}_{\Pi}) = 2 + 0 - 1 = 1$ or, in the case $z(\mathcal{L}) = 1$, that $z(\mathcal{L}_{\Pi}) = 1 + 0 - 0 = 1$. This fact proves the proposition in the case $n(\mathcal{L})=1$.\\
\indent We consider the case $n(\mathcal{L})=2$. From Proposition $\ref{teo12}$, it follows that $\theta<0$ and $\ker(\mathcal{L})=[\phi']$. Since $\mathcal{L}$ is a self-adjoint operator, it is invariant on the orthogonal complement subspace $\ker(\mathcal{L})^{\perp} = [\phi']^{\perp}$, that is, $\mathcal{L} : [\phi']^{\perp} \to [\phi']^{\perp}$. Moreover, because $\mathcal{L}$ is self-adjoint with closed range, we also have ${\rm range}(\mathcal{L}) = \ker(\mathcal{L})^{\perp}$. Consequently, there exists $h \in H_{\mathrm{per,e}}^{2}$ such that $\mathcal{L}h = 1$. Using the method of variation of parameters applied to the equation $\mathcal{L}h=1$, we see that $h$ can be expressed in terms of the fundamental set $\{y_1,y_2\}=\{y_1,\phi'\}$ as
\begin{equation}\label{varpar13}
	h(x)=y_1(x)\int_0^x\frac{\phi'(s)}{(c-\phi(s))W(y_1,y_2)(s)}ds-\phi'(x)\int_0^x \frac{y_1(s)}{(c-\phi(s))W(y_1,y_2)(s)}ds,
\end{equation}
where, $W(y_1,y_2)(s)$ indicates the Wronskian determinant of $y_1$ and $y_2$ that can be determined by using Abel's formula as
\begin{equation}\label{abel}
	W(y_1,y_2)(s)= \beta e^{\int_0^s\frac{\phi'(t)}{c-\phi(t)}dt}=\beta e^{-\int_0^s\frac{d}{dt}\ln(c-\phi(t))dt}=\beta\left(\frac{c-\phi(0)}{c-\phi(s)}\right),
\end{equation}	
where $\beta$ is a constant that can be assumed to be equal to one. Indeed, since $\{y_1,\phi'\}$ is a fundamental set of solutions of the homogeneous equation $\mathcal{L}h=0$, we deduce that $\beta \neq 0$. In addition, since the solution $\phi$ is even, it follows that $\phi'$ is odd and $\phi'(0)=0$, with $\phi''(0)\neq 0$, because $\phi'$ is non-trivial and solves the equation $\mathcal{L}\phi'=0$. By considering the normalized initial condition $y_1(0)=\frac{1}{\phi''(0)}$, we conclude that $W(y_1,\phi')(0)=1$, so that $\beta=1$, as required. By $(\ref{varpar13})$ and $(\ref{abel})$, we obtain a more convenient expression for the function $h$ as follows:
\begin{equation}\label{varpar1}
	h(x)=\frac{1}{c-\phi(0)}\left[(\phi(x)-\phi(0))y_1(x)-\phi'(x)\int_0^xy_1(s)ds\right].
\end{equation}
Since $\theta<0$, we obtain that $y_1$ is not periodic. Thus, by integration by parts, we obtain from $(\ref{varpar1})$
\begin{equation}\label{varpar2}
	\int_0^{2\pi}hdx=\frac{2}{c-\phi(0)}\left[-\phi(0)\int_0^{2\pi}y_1dx+\int_0^{2\pi} y_1\phi dx\right].
\end{equation}
\indent The next step is to obtain convenient expressions for the integrals $\int_0^{2\pi} y_1(x)dx$ and $\int_0^{2\pi} y_1(x)\phi(x)dx$. Indeed, multiplying the equation $\mathcal{L}h = 1$ by $y_1$, integrating over the interval $[0,2\pi]$, performing two integration by parts, and using the fact that $h$ is even and periodic, we deduce from $(\ref{varpar1})$ that
\begin{equation}\label{varpar3}
	\int_0^{2\pi}y_1dx=(c-\phi(0))h(0)y_1'(2\pi)=0.
\end{equation}	
\indent Again, since $\mathcal{L}:[\phi']^{\bot}\rightarrow [\phi']^{\bot}$ and $\phi$ is periodic and even, it follows that there exists $\chi\in H_{\rm per,e}^2$, such that $\mathcal{L}\chi=\phi$. By equation $(\ref{CHode})$, we have
\begin{equation}\label{2L}
	\mathcal{L}\phi=c(\phi''-\phi)+\omega\phi-2A.
\end{equation}
Using that $\mathcal{L}h=1$ and
\begin{equation}\label{3L}
	\mathcal{L}1=(c-\omega)-3\phi+\phi'',
\end{equation}
we conclude that $\chi$ can be expressed by
\begin{equation}\label{chi}
	\chi(x)=-\frac{c}{2c+\omega}+\frac{1}{2c+\omega}\phi(x)+\frac{c(c-\omega)+2A}{2c+\omega}h(x).
\end{equation}
Thus, multiplying the equation $\mathcal{L}\chi =\phi$ by $y_1$, integrating over the interval $[0,2\pi]$, performing two integration by parts, and using the fact that $\chi$ is even and periodic, we deduce from $(\ref{varpar1})$ and $(\ref{chi})$ that
\begin{equation}\label{varpar4}
	\int_0^{2\pi}y_1\phi dx=(c-\phi(0))\chi(0)y_1'(2\pi)=-\frac{(c-\phi(0))^2}{2c+\omega}y_1'(2\pi).
\end{equation}	
Combining the information from $(\ref{varpar2})$, $(\ref{varpar3})$, and $(\ref{varpar4})$, and using the fact that $c - \phi(0) > 0$, we conclude that
\begin{equation}\label{varpar5}
	\int_0^{2\pi}hdx=-2y_1'(2\pi)\frac{c-\phi(0)}{2c+\omega}.
\end{equation}	
\indent We need to calculate $y_1'(2\pi)$. In fact, first we need to use $(\ref{rel-varphi})$ and the fact that $\varphi_1(x)=\left(\frac{c-\phi(0)}{c-\phi(x)}\right)^{-1/2}y_1(x)$, where $\varphi_1$ is not periodic. Together with the $2\pi-$periodic function $\varphi_2(x)=\left(\frac{c-\phi(0)}{c-\phi(x)}\right)^{-1/2}\phi'(x)$, we have the fundamental set $\{\varphi_1,\varphi_2\}$ of solutions of the equation $\mathcal{M}_{\tau}w = 0$. Thus, we obtain by $(\ref{relpq1})$ and the explicit expressions of $\varphi_1$ and $\varphi_2$ that 
\begin{equation}\label{dery1}
	y_1'(2\pi)=\varphi_1'(2\pi)=\theta\varphi_2'(0)=\theta\phi''(0).
\end{equation}
Since $\phi''(0) < 0$ and $\theta < 0$, we obtain from $(\ref{dery1})$, that $y_1'(2\pi) > 0$. From $(\ref{varpar5})$ and the fact that $c - \phi(0) > 0$, it follows that $\langle \mathcal{L}^{-1}1, 1 \rangle = \int_0^{2\pi} hdx < 0$. By Remark $\ref{pel}$, this implies that
\begin{equation}\label{7L}
	n(\mathcal{L}_{\Pi}) = n(\mathcal{L}) - n_0 - z_0 = 2 - 1 - 0 = 1,
\end{equation}
and
\begin{equation}\label{8L}
	z(\mathcal{L}_{\Pi}) = z(\mathcal{L}) + z_0 = 1 - 0 = 1.
\end{equation}
The proposition is thus proved.

\end{proof}

\begin{remark}\label{kerLpi}
	Using the implicit function theorem together with Proposition $\ref{propR0}$, we ensure the existence of a smooth curve $c \mapsto \phi$ of periodic waves for all $c > \frac{\omega}{2}$ with the zero mean property. For the details of this argument, we refer the reader to \cite[Lemma 3.8]{NLP}.
\end{remark}

\begin{remark}  It is important to highlight that, by Remark $\ref{kerLpi}$, we can differentiate equation $(\ref{CHode})$ with respect to $c$ to obtain
\begin{equation}\label{1L}
	\mathcal{L}\left(\frac{d\phi}{dc}\right)=\phi''-\phi-\frac{dA}{dc}.
\end{equation}
\indent On the other hand, from equation $(\ref{CHode})$, we have
\begin{equation}\label{2L1}
	\mathcal{L}\phi=c(\phi''-\phi)+\omega\phi-2A.
\end{equation}
We also have
\begin{equation}\label{3L1}
	\mathcal{L}1=(c-\omega)-3\phi+\phi''.
\end{equation}
Combining the results obtained from  $(\ref{1L})$, $(\ref{2L1})$, and $(\ref{3L1})$, we obtain
\begin{equation}\label{4L}
	\mathcal{L}\left(\omega+2\phi+(-\omega-2c)\frac{d\phi}{dc}\right)=\omega(c-\omega)+(\omega+2c)\frac{dA}{dc}-4A=d_c.
\end{equation}
\indent We can employ some of the facts used in Proposition $\ref{propR0}$ to deduce important results. Indeed, the element $\omega+2\phi+(-\omega-2c)\frac{d\phi}{dc}$ is an even periodic function. In addition, it belongs to the kernel of $\mathcal{L}$ if and only if $d_c=0$. Moreover, if $d_{c}=0$, we have $z(\mathcal{L})=2$. Since $n(\mathcal{L}_{\Pi})=1$ and $z_0=n_0=0$, it follows that $n(\mathcal{L})=1$. If $d_c<0$, we obtain $z(\mathcal{L})=1$ and $\langle\mathcal{L}^{-1}1,1\rangle=\frac{2\pi\omega}{d_c}<0$. Since $n(\mathcal{L}_{\Pi})=1$, $n_0=1$, and $z_0=0$, it follows that $n(\mathcal{L})=2$. Finally, if $d_c>0$, we again obtain $z(\mathcal{L})=1$, and since $n_0=z_0=0$ and $n(\mathcal{L}_{\Pi})=1$, we conclude that $n(\mathcal{L})=1$.
\end{remark}

\begin{remark} We can calculate the exact sign of $d_c$ for small-amplitude periodic waves constructed in Proposition $\ref{propstokes1}$. Indeed, using the explicit expressions in $(\ref{varphi-stokes31})-(\ref{varphi-stokes3})$, we have
\begin{equation}\label{5L}
	d_c=-\frac{\omega^2}{6}+\mathcal{O}(a^2),
\end{equation}
where $a>0$ is sufficiently small. It follows from $(\ref{5L})$ that $d_c<0$ for all $c$ close to $\frac{\omega}{2}$ to the right. Moreover, we also obtain that $d_c\in \ker(\mathcal{L})^{\bot}$. Hence, for the small-amplitude periodic waves, it follows that
\begin{equation}\label{6L}
	\langle\mathcal{L}^{-1}1,1\rangle=\left\langle\frac{1}{d_c}\left(\omega+2\phi+(-\omega-2c)\frac{d\phi}{dc}\right),1\right\rangle=\frac{2\pi\omega}{d_c}<0.
\end{equation}

 Using Remark $\ref{pel}$, we observe that
\begin{equation}\label{7L1}
	n(\mathcal{L}_{\Pi}) = n(\mathcal{L}) - n_0 - z_0.
\end{equation}
Therefore, from $(\ref{6L})$ and $(\ref{7L1})$, since $n_0 = 1$, $z_0 = 0$, and $n(\mathcal{L}_{\Pi}) = 1$ (see Proposition $\ref{propR0}$), we conclude that $n(\mathcal{L}) = 2$ for the small-amplitude periodic waves.
\end{remark}

\section{Spectral stability of periodic waves - Proof of Theorem $\ref{theorem-stability}-(iii)$}
\label{sec-4}
\indent 

\begin{proposition}\label{dEphi}
	Let $c > \frac{\omega}{2}$ be fixed. If $\frac{d}{dc}E(\phi) > 0$, then the periodic wave $\phi$ is spectrally stable in the sense of Definition $\ref{defstab-spectralzm}$.
	\end{proposition}
\begin{proof}  
	Suppose that $\frac{d}{dc}E(\phi) > 0$. By Proposition $\ref{propR0}$, we have $n(\mathcal{L}_{\Pi}) = z(\mathcal{L}_{\Pi}) = 1$. Hence, since $\mathcal{L}_{\Pi}=\mathcal{L}+\frac{1}{2\pi}\langle\phi', \partial_x\cdot\rangle+\frac{3}{2\pi}\langle\phi, \cdot\rangle,$ we obtain by applying \cite[Proposition 3.8]{ANP} with $Q(u) = E(u)$ and $Q'(u) = u - u''$, that there exists a constant $C > 0$ such that
	\begin{equation}\label{coerc1}
		\langle \mathcal{L}v, v \rangle =\left\langle\mathcal{L}v+\frac{1}{2\pi}\langle\phi', \partial_xv\rangle+\frac{3}{2\pi}\langle\phi,v\rangle,v\right\rangle= \langle \mathcal{L}_{\Pi}v, v \rangle \geq C ||v||_{L_{{\rm per}}^2}^2,
	\end{equation}
	for all $v \in H_{\rm per, m}^2$ such that $\langle v, \phi - \phi'' \rangle = 0$ and $\langle v, \phi'\rangle=0$. Therefore, one has, since $z(\mathcal{L}_{\Pi})=1$, that
	\begin{equation}\label{coerc2}
		\langle \mathcal{L}v, v \rangle = \langle \mathcal{L}_{\Pi}v, v \rangle \geq0,
	\end{equation}
	for all $v \in H_{\rm per, m}^2$ such that $\langle v, \phi - \phi'' \rangle = 0$. The inequality in $(\ref{coerc2})$ establishes that $\mathcal{L}_{|_{\{{1, \phi - \phi''}\}^{\bot}}} \geq 0$, as desired, and the first part of Theorem \ref{theorem-stability}-(iii) is therefore proved for all possible cases of $d_c$, that is, $d_c \neq 0$ and $d_c = 0$.
	
\end{proof}

\indent A sufficient condition for ensuring that $\frac{d}{dc}E(\phi) > 0$ is established in the next result.
\begin{proposition}\label{propstab}
	Let $c > \frac{\omega}{2}$ be fixed. If $d_c \neq 0$ and $\frac{dA}{dc} > 0$, then the wave $\phi$ is spectrally stable in the sense of Definition $\ref{defstab-spectralzm}$. In particular, we have that $\frac{d}{dc}E(\phi) > 0$.
\end{proposition}
\begin{proof} 	Since $d_c\neq0$, we have that $\ker(\mathcal{L}) = [\phi']$. The symmetric matrix $\mathcal{A}(0)$ given by \eqref{matrixA} is given by
		\begin{equation}\label{matrixD}\mathcal{A}(0)=\left[\begin{array}{llll}
			\langle\mathcal{L}^{-1}(\phi-\phi''),\phi-\phi''\rangle &  \langle\mathcal{L}^{-1}(\phi-\phi''),1\rangle\\\\
			\langle\mathcal{L}^{-1}(\phi-\phi''),1\rangle &  \langle\mathcal{L}^{-1}1,1\rangle
			\end{array}\right].
	\end{equation}
To clarify for the reader, in Proposition $\ref{prop-Pel}$ we have, in this context, $z_0=\dim(\ker(\mathcal{A}(0)))$ and $n_0=n(\mathcal{A}(0))$. We prove our result for the case $d_c\neq0$. Again, since $\mathcal{L}:[\phi']^{\bot}\rightarrow [\phi']^{\bot}$, there exists $\Phi\in H_{\rm per,e}^2$, such that $\mathcal{L}\Phi=-\phi''$. Since $\mathcal{L}h=1$, we see by $(\ref{6L})$ that $\int_0^{2\pi}hdx=\langle\mathcal{L}^{-1}1,1\rangle=\frac{2\pi\omega}{d_c}\neq0$. Thus, by $(\ref{2L})$ and $(\ref{3L})$, we obtain that $\Phi$ is given by
\begin{equation}\label{Phi}
\Phi(x)=\frac{c-\omega}{2c+\omega}-\frac{3\phi}{2c+\omega}-\frac{(c-\omega)^2+6A}{2c+\omega}h.
\end{equation}
Thus, we obtain by $(\ref{chi})$ and $(\ref{Phi})$
\begin{equation}\label{Lphiphi2}
\chi+\Phi=\mathcal{L}^{-1}(\phi-\phi'')=-\frac{\omega}{2c+\omega}-\frac{2\phi}{2c+\omega}+\frac{c\omega-\omega^2-4A}{2c+\omega}\mathcal{L}^{-1}1,
\end{equation}
so that 
\begin{equation}\label{Lphi1}
\langle\mathcal{L}^{-1}(\phi-\phi''),1\rangle=-\frac{\omega}{2c+\omega}+\frac{c\omega-\omega^2-4A}{2c+\omega}\langle\mathcal{L}^{-1}1,1\rangle.
\end{equation}
Gathering the results in $(\ref{Lphiphi2})$ and $(\ref{Lphi1})$, we obtain from $(\ref{matrixD})$
\begin{equation}\label{detD}\begin{array}{lllll}
\det \mathcal{A}(0)&=&\displaystyle-\frac{2}{2c+\omega}\int_0^{2\pi}(\phi^2+\phi'^2)dx\int_0^{2\pi}hdx\\\\
&+&\displaystyle\frac{2\pi\omega}{2c+\omega}\left[-\frac{2\pi\omega}{2c+\omega}+\frac{c\omega-\omega^2-4A}{2c+\omega}\int_0^{2\pi}hdx\right].\end{array}
\end{equation}
\indent We need to determine the sign of the second term on the right-hand side of~\eqref{detD}, namely the sign of
\begin{equation}\label{chiphi}\int_0^{2\pi}(\chi+\Phi)dx=-\frac{2\pi\omega}{2c+\omega}+\frac{c\omega-\omega^2-4A}{2c+\omega}\int_0^{2\pi}hdx.
\end{equation}
By Remark $\ref{kerLpi}$, the mapping $c\in \left(\frac{\omega}{2},+\infty\right)\mapsto \phi\in H_{\rm per,m,e}^{\infty},$ is smooth. Consequently, by $(\ref{1L})$ and $(\ref{Lphiphi2})$, it follows that
\begin{equation}\label{chiphiderA1}
\mathcal{L}\left(\chi+\Phi+\frac{d\phi}{dc}\right)=-\frac{dA}{dc}.
\end{equation}
From $(\ref{chiphiderA1})$ and using the fact that $\frac{dA}{dc}>0$, we then obtain
\begin{equation}\label{chiphiderA}
\chi+\Phi+\frac{d\phi}{dc}=-\frac{dA}{dc}\mathcal{L}^{-1}1.
\end{equation}
Integrating the result in $(\ref{chiphiderA})$ over the interval $[0,2\pi]$, and using the facts that $\frac{dA}{dc}>0$ and $\int_0^{2\pi}hdx\neq0$, we obtain that
\begin{equation}\label{intchiphiderA}
\int_0^{2\pi}(\chi+\Phi)dx=-\frac{dA}{dc}\langle\mathcal{L}^{-1}1,1\rangle=-\frac{dA}{dc}\int_0^{2\pi}hdx\neq0.
\end{equation}
\indent On the other hand, since we have $\frac{dA}{dc}>0$ for all $c>\frac{\omega}{2}$, it follows from $(\ref{detD})$ and by assuming $d_c<0$, that $\det \mathcal{A}(0)>0$. Again, since $\langle\mathcal{L}^{-1}1,1\rangle=\int_0^{2\pi}hdx<0$, we have $z_0=0$ and $n_0=2$. By Proposition $\ref{prop-Pel}$, we obtain
$$n\left(\mathcal{L}_{|_{\{1,\phi-\phi''\}^{\bot}}}\right)=n(\mathcal{L})-z_0-n_0=2-0-2=0.$$
Now, if $d_c>0$, we conclude $\int_0^{2\pi}hdx>0$. Since $\frac{dA}{dc}>0$, it follows from $(\ref{intchiphiderA})$ that $\int_0^{2\pi}(\chi+\Phi)dx<0$. Therefore, we obtain by $(\ref{detD})$ that $\det \mathcal{A}(0)<0$. By Proposition $\ref{prop-Pel}$, we obtain
$$n\left(\mathcal{L}_{|_{\{1,\phi-\phi''\}^{\bot}}}\right)=n(\mathcal{L})-z_0-n_0=1-0-1=0.$$
The spectral stability is now complete for the case $d_c\neq0$ and $\frac{dA}{dc}>0$.\\
 \indent To complete the proof of the proposition, we need to highlight some additional facts. In fact, by Remark $\ref{kerLpi}$, we have that the mapping $c\in\left(\frac{\omega}{2},+\infty\right)\mapsto \phi \in H_{\rm per,m,e}^{\infty}$ is smooth. Therefore, since $d_c \neq 0$, we obtain that $z(\mathcal{L}) = 1$ and that $\mathcal{L} : [\phi']^{\bot} \to [\phi']^{\bot}$ is invertible. Thus, by $(\ref{1L})$, we obtain that
 \begin{equation}\label{1L1}
 	\langle\mathcal{L}^{-1}(\phi-\phi''),\phi-\phi''\rangle=-\frac{1}{2}\frac{d}{dc}\int_0^{2\pi}(\phi^2+\phi'^2)dx-\frac{dA}{dc}\langle \mathcal{L}^{-1}(\phi-\phi''),1\rangle.
 \end{equation}
 Similarly, again by $(\ref{1L})$, we deduce from the fact that $\frac{d\phi}{dc}$ has the zero-mean property that
 \begin{equation}\label{2L2}
 	\langle\mathcal{L}^{-1}(\phi-\phi''),1\rangle=-\frac{dA}{dc}\langle \mathcal{L}^{-1}1,1\rangle.
 \end{equation}
  Thus, by $(\ref{1L1})$ and $(\ref{2L2})$, we obtain
\begin{equation}\label{deta0}
\det \mathcal{A}(0) = -\frac{1}{2}\left( \frac{d}{dc} \int_{0}^{2\pi} (\phi^2 + \phi'^2)dx \right)\langle \mathcal{L}^{-1}1, 1 \rangle=-\frac{d}{dc}E(\phi)\langle \mathcal{L}^{-1}1, 1 \rangle.
\end{equation}
As proved  below $(\ref{intchiphiderA})$, if $d_c < 0$ then
$\langle \mathcal{L}^{-1}1, 1 \rangle < 0$ and $\frac{dA}{dc} > 0$. Therefore, $(\ref{deta0})$ implies the well-known Vakhitov-Kolokolov stability criterion:
\begin{equation}\label{LPiphiphi2}
\langle \mathcal{L}_{\Pi}^{-1}(\phi - \phi''), \phi - \phi'' \rangle
= -\frac{d}{dc}E(\phi) < 0.
\end{equation}
Similarly, if $d_c > 0$, we obtain $\langle \mathcal{L}^{-1}1, 1 \rangle=\frac{2\pi\omega}{d_c}=\int_0^{2\pi}hdx > 0$ and $\det \mathcal{A}(0) < 0$. Thus, the same condition in $(\ref{LPiphiphi2})$ is also satisfied, and we conclude that if $d_c \neq 0$ and $\frac{dA}{dc} > 0$, then it follows that $\frac{d}{dc}E(\phi) > 0$.
\end{proof}

\subsection{A remark on the orbital stability of periodic waves}
\label{sec-5}

To finish, we prove the orbital stability of periodic waves. To this end, we follow the approach of \cite{ANP} (see also \cite{GrillakisShatahStraussI}).

\begin{proposition}\label{prop-Natali} Let $c>\frac{\omega}{2}$ be fixed. If $\frac{d}{dc}E(\phi)>0$, the periodic wave $\phi \in H^{\infty}_{\rm per,m,e}$ is orbitally stable in $H_{\rm per,m}^1$ in the sense of Definition $\ref{defstab}$.	
\end{proposition}
\begin{proof}
 By Propositions $\ref{propR0}$, we obtain that the linearized operator $\mathcal{L}_{\Pi}$ has a simple negative eigenvalue and a simple zero eigenvalue associated with the eigenfunction $\phi'$.\\
 \indent Following again the notation in \cite[Proposition 3.8]{ANP}, we consider $Q(u) = E(u)$, where $E(u)$ is given by $(\ref{Eu})$. Since $Q'(u) = u - u''$ and $\frac{d}{dc}E(\phi)>0$, we deduce from $(\ref{LPiphiphi2})$ that
$\langle \mathcal{L}_{\Pi}^{-1}(\phi - \phi''), \phi - \phi'' \rangle < 0$ for all $c > \tfrac{\omega}{2}$. Thus, there exists $C > 0$ such that
 \begin{equation}\label{coerc}
 \langle \mathcal{L}_{\Pi}v, v \rangle \geq C ||v||_{L_{\text{per}}^2}^2,
 \end{equation}
for all $v \in H_{\rm per, m}^2$ such that $\langle v, Q'(u) \rangle =  0$ and $\langle v, \phi' \rangle = 0$. Thus, by \cite[Theorem 3.6]{ANP}, we conclude that $\phi$ is orbitally stable in $H_{\rm per,m}^1$ the sense of Definition $\ref{defstab}$. This result completes the proof of Theorem $\ref{theorem-stability}$-(iii).
\end{proof}

\begin{corollary} Let $c>\frac{\omega}{2}$ be fixed. If $d_c>0$ and $\frac{dA}{dc}>0$, the periodic wave $\phi \in H^{\infty}_{\rm per,m,e}$ is orbitally stable in $H_{\rm per}^1$ in the sense of Definition $\ref{defstab}$.
\end{corollary}
\begin{proof} Since $d_c > 0$, we obtain that $n(\mathcal{L}) = z(\mathcal{L}) = 1$. In addition, since $\frac{dA}{dc} > 0$, we conclude by Proposition $\ref{propstab}$ that $\frac{d}{dc}E(\phi)>0$. Thus,
	\begin{equation}\label{derL}
		\left\langle\mathcal{L}\left(-\frac{d\phi}{dc}\right),-\frac{d\phi}{dc}\right\rangle=\left\langle \phi-\phi''+\frac{dA}{dc},-\frac{d\phi}{dc}\right\rangle=-\frac{d}{dc}E(\phi) < 0.
		\end{equation} 
	Therefore, again by \cite[Theorem 3.6]{ANP}, now applied to the linearized operator $\mathcal{L}$, we conclude that $\phi$ is orbitally stable in $H_{\rm per}^1$ in the sense of Definition $\ref{defstab}$. 
\end{proof}

\begin{remark}
	\label{remark-orbital}
	An important fact needs to be clarified. The notion of orbital stability in Definition~\ref{defstab} requires the existence of global solutions $u \in C(\mathbb{R}, H^1_{\rm per,m})$. For $s > \frac{3}{2}$, local solutions $u \in C((-t_0, t_0), H^s_{\rm per,m})$ for some $t_0 > 0$ exist due to the local well-posedness theory in \cite{CE-1998}, \cite{D}, \cite{hakka}, and \cite{HM2002}. By combining the local solution in $H_{\rm per,m}^s$ for $s > \frac{3}{2}$ with the conservation laws $M(u(t))=M(u_0)$ and $E(u(t)) = E(u_0)$ for all $t \geq 0$, we can extend it to a global solution in $H_{\rm per,m}^1$. Then, this global solution remains close to the smooth periodic waves in accordance with Definition \ref{defstab}. The same arguments can be applied for full the energy space $H_{\rm per}^1$ instead of $H_{\rm per,m}^1$.
\end{remark}

\begin{remark} Equation $(\ref{CHode})$ can be solved numerically using Python and by using the fact that $c \in \left(\frac{\omega}{2}, +\infty\right) \mapsto \phi$ is smooth with respect to $c$, as stated in Remark $\ref{kerLpi}$. This allows us to analyze the behaviour of the function $E(\phi)$ in terms of $c\in\left(\frac{\omega}{2},+\infty\right)$. We present plots of $E(\phi)$ for the cases $\omega = 1$, $\omega = 2$, $\omega = 3$, and $\omega = 5$. For other values of $\omega$, the behavior of $E(\phi)$ for all $c > \frac{\omega}{2}$ is similar. From our analysis, we obtain that $\frac{d}{dc}E(\phi) > 0$ for all $c > \frac{\omega}{2}$. By Theorem \ref{theorem-stability}-(iii), and supported by numerical computations, we conclude that $\phi$ is spectrally and orbitally stable. It is important to mention that even though the function $E(\phi)$ is strictly increasing in terms of $c$, the function $d_c$ in $(\ref{4L})$ may still vanish at points in the interval $c\in\left(\frac{\omega}{2},+\infty\right)$\footnote{This behaviour is similar to that reported in \cite[Remark 11]{GMNP}, where the authors studied the monotonicity of the period map with respect to the energy levels and proved that the period map may have at most one non-degenerate critical point.}. However, the possible existence of zeroes of $d_c$ is not important for the spectral and orbital stability. 
	\end{remark}
	
	\begin{figure}[h]
		\begin{minipage}[t]{0.4\linewidth}
			\includegraphics[width=3.2in]{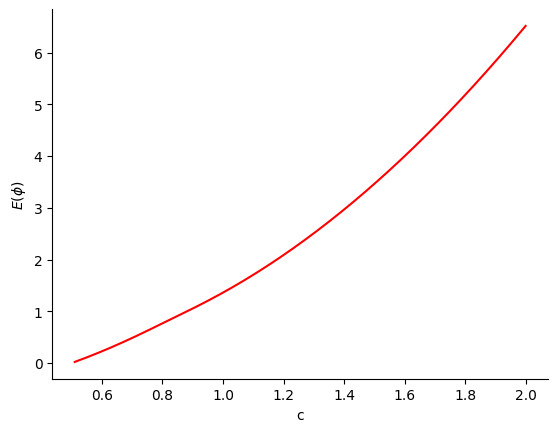}
		\end{minipage}
		\hspace{30pt}
		\begin{minipage}[t]{0.40\linewidth}
			\includegraphics[width=3.2in]{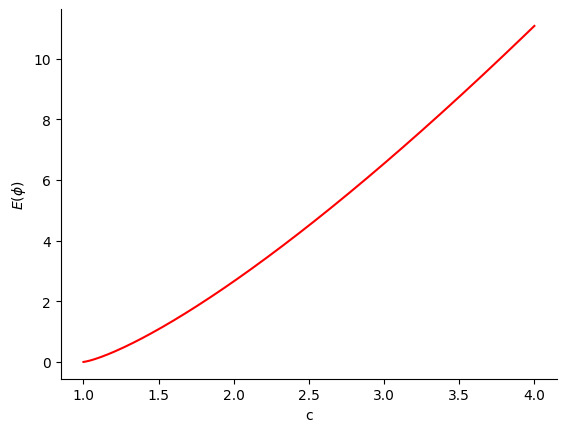}
		\end{minipage}
		\caption{Graph of $E(\phi)$ for $\omega=1$ (left), and graph of $A$ for $\omega=2$ (right).}
		\label{exact_p2E}
	\end{figure}

	\begin{figure}[h]
		\begin{minipage}[t]{0.4\linewidth}
			\includegraphics[width=3.2in]{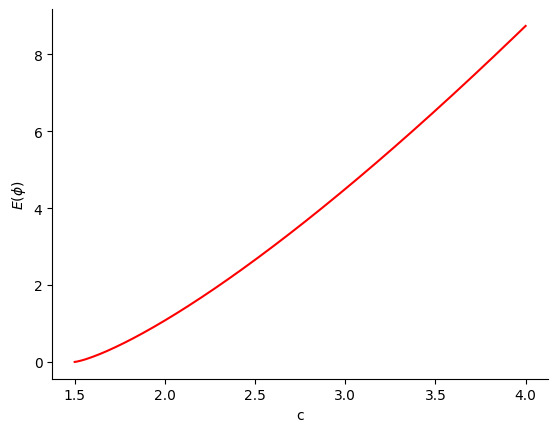}
		\end{minipage}
		\hspace{30pt}
		\begin{minipage}[t]{0.40\linewidth}
			\includegraphics[width=3.2in]{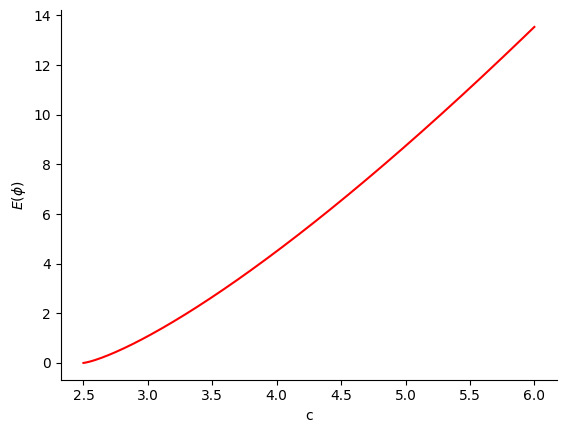}
		\end{minipage}
		\caption{Graph of $E(\phi)$ for $\omega=3$ (left), and graph of $A$ for $\omega=5$ (right).}
		\label{exact_p4E}
	\end{figure}

\newpage
\section*{Acknowledgments}
The author would like to thank J.C. Bronski and M.A. Johnson for fruitful discussions that improved the presentation of this work. F. Natali is partially supported by CNPq/Brazil (grant 303907/2021-5).

\end{document}